\documentclass[11pt,twoside]{amsart}
\usepackage{amsmath, amsthm, amscd, amsfonts, amssymb, graphicx, color}
\usepackage[bookmarksnumbered, colorlinks, plainpages]{hyperref}

\usepackage{amsmath}
\usepackage{amsfonts}

\newtheorem{theo}{Theorem}[section]
\newtheorem{pro}[theo]{Proposition}
\newtheorem{coro}[theo]{Corollary}
\newtheorem{lem}[theo]{Lemma}

\theoremstyle{definition}
\newtheorem{exam}[theo]{Example}
\newtheorem{rem}[theo]{Remark}
\newtheorem{defi}[theo]{Definition}

\newtheorem{nota}[theo]{Notation}
\newtheorem{que}[theo]{Question}

\newcommand{\ndim}{\mbox{{\it n}-{\rm dim}}}
\newcommand{\kdim}{\mbox{{\it k}-{\rm dim}}}

\newcommand{\be}{\begin{enumerate}}
\newcommand{\ee}{\end{enumerate}}
\author{S.M. Javdannezhad}
\address{Sayed Malek Javdannezhad, Department of Science, Shahid Rajaee Teacher Training University: Tehran, IR}
\email{sm.javdannezhad@gmail.com}

\author{M. Maschizadeh}
\address{Mohammad Maschizadeh, Department of Mathematics, Shahid Chamran University of Ahvaz,  Ahvaz, Iran}
\email{m. maschi96@gmail.com}

\author{N. Shirali}
\address{Nasrin Shirali, Department of Mathematics, Shahid Chamran University of Ahvaz,  Ahvaz, Iran}
\email{shirali\_n@scu.ac.ir\\nasshirali@gmail.com}

\title[On $AB5^*$ modules with Noetherian dimension]{On $AB5^*$ modules with Noetherian dimension}

\begin{document}

\thanks{{\scriptsize
        \hskip -0.4 true cm MSC(2010): Primary: 16P60, 16P20 ; Secondary:
        16P40.
        }}


\begin{abstract}
In this paper, we study the Noetherian dimension of sum of certain modules.  It is proved that for any module $M$ which is an irredundant sum of submodules, each of which has Noetherian dimension $\leq \alpha$, if $M$ has finite spanning dimension ($fsd$-module, for short) or it is a weakly atomic module, then $\ndim\,M \leq \alpha$. Here, by a weakly atomic module we mean a module $M$ for which every proper non-small submodule $N$, has Noetherian dimension strictly less than that of $M$. Also, it is proved that if $M$ is an $AB5^*$ module with Noetherian dimension and  $\{N_i\}$  is a family of submodules of $M$ such that $\ndim\,\frac{M}{N_i} \leq \alpha$, for each $i$,  then $\ndim\,\frac{M}{\cap N_i} \leq \alpha$.  Using this, we give a structure theorem for $\alpha$-short modules in the category of $AB5^*$ and finally, we classify $\alpha$-short modules in this category.
\end{abstract}

\maketitle
\section{Introduction}
 Throughout, all rings are associative with identity and, unless stated otherwise, all modules are unital right modules. Let $R$ be a ring and let $M$ be an $R$-module. $N \ll M$ will donote $N$ is an  $N$ is a small submodule of $M$. Also $\kdim\,M$ and $\ndim\,M$ denote respectively the Krull dimension and the Noetherian dimension of $M$. The existence of these two dimensions are equivalent for any module $M$, see \cite{le2} . The latter dimension is also called the dual Krull dimension in many other articles. For more detailes on the Krull and the Noetherian dimension, the reader is referred to  \cite{al-ri, al-sm,   go-ro, ka1, le2}. It is convenient, when we are dealing with these dimensions, to begin our list of ordinals with $-1$.\\
 The purpose of this paper is to study some natural unquestioned results and properties concerning the Noetherian dimension of rings and modules. Although many studies and advances have been made in this field, but some trivial and natural questions in this area remain still unanswered. For example, one of
 the basic results concerning a module with Krull dimension due to Gordon, Robson and Lenagan is as follows and in order to honor them,
 we call it $GRL$-Theorem, see \cite{go-ro, len}.

 \begin{theo}\label{grl}
 [GRL] Let an $R$-module $M$ have Krull dimension and also be the sum of submodules each of which has Krull dimension $\leq \alpha$. Then $\kdim\,M \leq \alpha$.
 \end{theo}

 The natural dual of this fact is as follows: Let $M$ be an $R$-module with  Noetherian dimension and  $\{N_i\}$  a family of submodules of $M$ such that $\ndim\,\frac{M}{N_i} \leq \alpha$, for each $i$, then $\ndim\,\frac{M}{\cap N_i} \leq \alpha$. Here, it is natural to be wondered about the statements of Theorem \ref{grl} and its dual for Noetherian dimension, and we are led to ask the following questions.

\begin{que}\label{q1}
    Is the statement of Theorem \ref{grl} is also valid for Noetherian dimension?
\end{que}
\begin{que}\label{q2}
 Is the dual of Theorem \ref{grl} holds for Noetherian dimension?
\end{que}

This paper is organized as follows. In Section $2$, as the main aim, we answer to both of the above questions. Although Example \ref{zp} shows that the answer to Question \ref{q1}  is not "Yes" in general. However, we shall see that it is "Yes" for certain modules which are an irredundant sum of their submodules. We show that if $M$ is an irredundant sum of submodules with Noetherian dimension $\leq \alpha$ and $\ndim\,M > \alpha$, then $M$ necessarily has an infinite strictly descending chain of non-small submodules with Noetherian dimension equal to   $\ndim\,M$.  From this, we see that if such a module $M$ has finite spanning dimension or it is weakly-atomic (i.e., every non-small submodules of $M$ has Noetherian dimension strictly less than $\ndim\,M$), then $\ndim\,M \leq \alpha$. Also by an example we show that the answer to Question \ref{q2} is not "Yes" in general and we shall see that in caes $M$ satisfies in the property of $AB5^*$, then the answer is "Yes". Using these facts, in Section $3$, we give a structure theorem for $\alpha$-short modules in the category of $AB5^*$ and then we classify these modules. We show that for any module $M$ in the category of $AB5^*$, $M \in \mathcal{A}$ (that is, the $\alpha$-shortness and the Noetherian dimension of $M$ are equal) if and only if $\ndim\,M$ is a limit ordinal or $\ndim\,A \leq \ndim\,\frac{M}{A}$, for any atomic submodule $A$ of $M$.

\section{ GRL theorem and its dual for Noetherian dimension}

We  begin by focus on Question \ref{q1}. First, we give a
counterexample to show that the answer to Question \ref{q1} is not
'Yes" in general.

\begin{exam}\label{zp}
 The $p$-prufer group $M= \mathbb{Z}_{p^\infty}$ is a uniserial $\mathbb Z$-module and $$H_0 \subseteq H_1 \subseteq H_2 \subseteq \dots H_n \subseteq H_{n+1} \subseteq \dots$$ is the uniqe chain of its submodule, where  $H_n = <\frac{1}{p^n} + \mathbb Z>$, for each integer $n \geq 0$. Note that $\mathbb Z_{p^{\infty}} = \sum_{n \geq 0}H_n$. It is easy to verify that $\ndim\,H_n = 0$, that is every $H_n$ is Noetherian . However $\ndim\,M =1$.
\end{exam}

 If $M =\sum_\Delta  M_\delta$, then this sum is called irredundant if, for every $\delta_0$, $ M \ne \sum_{\delta \ne \delta_0} M_\delta $. Note that to represent a  module as an irredundant sum of its submodules is not the case in general. For example uniserial modules and local modules are not representable as an irredundant sum of their submodules.

\begin{pro}\label{irr}
    Let  $M$ be an irredundant sum of submodules each of which has Noetherian dimension $\leq \alpha$. If $\ndim\,M > \alpha$, then $M$ has an infinite descending chain $A_1 \supsetneq A_2 \supsetneq A_3 \supsetneq \dots$ of non-small submodules such that $\ndim\,A_i = \ndim\,M$.
\end{pro}

\begin{proof}
    Let  $M=\sum_{i \in \Delta}N_i$ be an irredundant sum with  $\ndim\,N_i \leq \alpha$.  If $\ndim\,M > \alpha$, then $\ndim\,\frac{M}{N_i} > \alpha$ for each $i$. Set $A_1 = \varSigma_{j\neq 1} N_j$, then $N_1 \nsubseteq A_1$ (note, if $N_1 \subseteq A_1$ then $N_1$ may be omited) and so $N_1 \cap A_1 \subsetneq N_1$. This implies that $0 \neq \frac{N_1}{N_1 \cap A_1} \cong \frac{N_1 + A_1}{A_1} = \frac{M}{A_1}$. So $\ndim\,\frac{M}{A_1} \leq \alpha$ and consequently, $\ndim\,A_1 > \alpha$. Continue this manner and by the way of induction, we obtain an infinite descending chain  $A_1 \supsetneq A_2 \supsetneq A_3 \supsetneq \dots$   of non-small submodules, such that $\ndim\,\frac{M}{A_i} \leq \alpha$. Hence, $\ndim\,A_i = \ndim\,M$ and we are done.
\end{proof}

  Recall that an $R$-module $M$ has finite spanning dimension ($fsd$-module, for short) if for every descending chain  $N_1\supseteq N_2\supseteq N_3 \supseteq \dots$ of submodules $M$, there is a number $n$ such that $N_i \ll M$ for all $i > n$, or equvalently $M$ satisfies the DCC on non-small submodules. Artinian and hollow modules are examples for $fsd$-modules. From Theorem \cite[3.1]{fle}, it follows that every $fsd$-module $M$ is an irredundant finite sum of hollow submodules. Hence  $\ndim\,M = \ndim\,H$ for some hollow submodule $H$ of $M$. In what follows, we give the generalization of this fact for every arbitrary representation of $M$ as an irredundant sum of submodules.

\begin{theo}
    Let an $fsd$-module
 $M$ be an  irredundant sum of submodules each of which has Noetherian dimension $\leq \alpha$.  Then $\ndim\,M \leq \alpha$.
\end{theo}

\begin{proof}
    Since $M$ is an $fsd$-module, it has not any infinite descending chain of non-small submodules. Hence we are done, by Proposition \ref{irr}.
\end{proof}

 \begin{defi}
 An $R$-module $M$ is called  weakly atomic if $\ndim\,N < \ndim\,M $, for every proper non-small submodule $N$ of $M$.
 \end{defi}

\begin{rem}
Since every submodule of an atomic module is small, we conclude that every atomic module is a weakly atomic module, but not in vice versa. For example, let $M$ be a local module (i.e., if it has exactly one maximal submodule that contains all proper submodules) with the unique maximal submodule $K$ such that $\ndim\,M \geq 1$.  Since all proper submodules of $M$ are small, $M$ is weakly atomic. Furthermore, since $M$ has Noetherian dimension $ \geq 1$, $\ndim\,M = \sup\{\ndim\,K, \ndim\,\frac{M}{K}\} = \ndim\,K$ and this shows that $M$ is not atomic.
\end{rem}

\begin{pro}
    Let a weakly atomic module $M$ be an  irredundant sum of submodules each of which has Noetherian dimension $\leq \alpha$.  Then  $\ndim\,M \leq \alpha$.
 \end{pro}

 \begin{theo}
 Let $M$ be a weakly atomic module with $J(M) \ll M$. If $M$ is a sum of hollow submodules each of which has Noetherian dimension  $\leq \alpha$. Then $\ndim\,M \leq \alpha$.
 \end{theo}

 \begin{proof}
 Since $J(M) \ll M$, by \cite[41.5]{wis} we may assume that $M$ is an irredundant sum of submodules. From the previous theorem, we conclude $\ndim\,M \leq \alpha$ and we are done.
 \end{proof}

  Now we are going to focus on the Question \ref{q2}. First we give an example in order to shows that the answer to \ref{q2} is not "Yes"  in general.

\begin{exam}
  Let  $R = \Pi_{i\in \mathbb N}F_i = F_1 \times F_2 \times \dots $, where  each $F_i$ is a filed. Then  $J(R) =0$ and clearly $R$ is not Noetherian. Note that for every maximal ideal $M \in max(R)$, we have $\ndim\,\frac{R}{M} = 0$. On the other hand,  $ \frac{R}{\cap max(R)} = \frac{R}{J(R)} \cong R $ is not Noetherian. This shows that the dual of Theorem \ref{grl} does not hold. Here, we emphasize that $R$ not only is non-Noetherian, but also has not Krull (resp., Noetherian) dimension, too. For this, note that it has no finite Goldie dimension, see \cite[Proposition 1.4]{go-ro}.
\end{exam}

We cite the following fact from \cite[Proposition 4.5]{go-ro} or \cite{len}.

\begin{lem}\label{l2}
    If an $R$-module $M$ has infinite chains
    $$M= A_0 \supseteq A_1 \supseteq A_2 \supseteq A_3\supseteq ...\quad
    \text{and} \quad 0 = B_0 \subseteq B_1 \subseteq B_2 \subseteq B_3 \subseteq \dots$$
    such that $A_n \cap B_{n+1} \nsubseteq A_{n+1} + B_n$, for each $n \geq 0$, then $M$ does not have Krull dimension.
\end{lem}

 \begin{defi}
    An $R$-module $M$  is said to be an $AB5^*$ module, if for every submodule $B$ and inverse system $\{A_{i}\}_{i\in I}$ of submodules of $M$,
    $$B +\bigcap_{i\in I}A_{i}=\bigcap_{i\in I}(B+A_{i}).$$
\end{defi}

 Note that though most modules $M$ don't have this property, but fortunately Artinian modules and linearly compact modules satisfy the property of $AB5^{*}$, see \cite[29.8]{wis}. In what follows, we show that the dual of the Theorem \ref{grl},  holds for Noetherian dimension of  $AB5^*$ modules.

\begin{theo}\label{t2}
Let $M$ be an $AB5^*$ module with Noetherian dimension and let $\{N_i\}_{i \in \Delta}$  a family of submodules of $M$ such that $\ndim\,\frac{M}{N_i} \leq \alpha$, for each $i$. Then $\ndim\,\frac{M}{\cap N_i} \leq \alpha$.
\end{theo}

\begin{proof}
    Set $N = \cap_{i \in \Delta} N_i$ and let $0 \neq \frac{A}{N}$ be a non-zero submodule of $\frac{M}{N}$, then $N \subsetneq A$, so there exists an indexed submodule, $N_1$ say, such that $A \nsubseteq N_1$. Thus $N \subseteq A \cap N_1 \subsetneq A$ and hence $0 \neq   \frac{A}{A \cap N_1} \cong \frac{A+N_1}{N_1} $
    and $\ndim\,\frac{A}{A \cap N_1} \leq \alpha$. This shows that each submodule $A$ of $M$, properly containing $N$, contains a proper submodule $A'$ containing $N$ and $\ndim\,\frac{A}{A'} \leq \alpha$. Thus, there exists a descending chain
    $M = A_0 \supsetneq A_1 \supsetneq A_2 \supsetneq \dots \supseteq N$ with ordinal indexes such that for each ordinal $\lambda$, $\ndim\,\frac{A_\lambda}{A_{\lambda + 1}} \leq \alpha$. For limit ordinal $\lambda$, define $A_\lambda = \cap_{\delta < \lambda} A_\delta$. Since $M$ is a set, there exists an ordinal number $\lambda$ such that $A_\lambda = N$. Now we proceed by induction on $\lambda$ to show that
    $$P(\lambda) : \ndim\,\frac{M}{A_\lambda} \leq \alpha.$$
    Clearly $P(0)$ and also  $P(\lambda)$, for every non-limit ordinal number $\lambda$, since\\
    $\ndim\,\frac{M}{A_\lambda} = sup\{\ndim\,\frac{M}{A_{\lambda -1}}, \ndim\,\frac{A_{\lambda-1}}{A_\lambda}\}$. Now let $\lambda$ be a limit ordinal number. If $\ndim\,\frac{M}{A_\lambda} \nleq \alpha$, then there exists an ascending chain
    $A_\lambda = N = B_0 \subsetneq B_1 \subsetneq B_2 \subsetneq \dots $
    such that $\ndim\,\frac{B_{i+1}}{B_i} \geq \alpha$.
    Note that $B_1 \nsubseteq N$ , so there exists $\delta < \lambda$ such that $B_1 \nsubseteq A_\delta$. Set $A_\delta = A_1$, then $B_1 \nsubseteq A_1$ and so $A_0 \cap B_1 \nsubseteq A_1 + B_0$. Suppose that $A_i \cap B_{i+1} \nsubseteq A_{i+1} + B_i$, for each $0 \leq i \leq n-1 < \lambda$. In particular, $A_{n-1} \cap B_n \nsubseteq A_n + B_{n-1}$. We seek $A_{n+1}$ and $B_{n+1}$ such that $A_n \cap B_{n+1} \nsubseteq A_{n+1} + B_n$. If for each $k > 0$, $A_n \cap B_{n+k} \subseteq B_n$,  then $A_n \cap B_{n+k} \subseteq B_{n+k-1}$. Then by modular laws, it follows that that $B_{n+k-1} = (A_n \cap B_{n+k}) + B_{n+k-1} = (A_n + B_{n+k-1}) \cap B_{n+k}$. Therefore
    $\frac{B_{n+k}}{B_{n+k-1}} \cong \frac{B_{n+k}}{B_{n+k} \cap (A_n + B_{n+k-1})} \cong  \frac{A_n + B_{n+k}}{A_n + B_{n+k-1}} \cong \frac{(A_n + B_{n+k})/A_n}{(A_n + B_{n+k-1})/A_n)}$.
    This shows that
     $\ndim\,\frac {(A_n + B_{n+k})/A_n}{(A_n + B_{n+k-1})/A_n)} \geq \alpha$. Hence $\frac{A_n+B_n}{A_n} \subseteq \frac{A_n + B_{n+1}}{A_n} \subseteq \frac{A_n + B_{n+2}}{A_n} \subseteq \dots $
    is an ascending chain in $\frac{M}{A_n}$ such that every quotient has Noetherian dimension at least $\alpha$, contradiction to $\ndim\,\frac{M}{A_n} \leq \alpha$. Consequently, there exists $k>0$ such that $A_n \cap B_{n+k} \nsubseteq B_n$. Set $B_{n+k} = B_{n+1}$ and conclude $A_n \cap B_{n+1} \nsubseteq B_n$.
    But $B_n = B_n + N = B_n + A_\lambda$, so
    $A_n \cap B_{n+1} \nsubseteq B_n + \cap_{\delta<\lambda}A_\delta$. Since $M$ has the property of $AB5^*$, we have $B_n + \cap_{\delta<\lambda}A_\delta = \cap_{\delta<\lambda}(B_n + A_\delta)$, consequently there exists $\delta > n$ that $A_n \cap B_{n+1} \nsubseteq A_\delta + B_n$. Set $A_\delta = A_{n+1}$ and conclude that $A_n \cap B_{n+1} \nsubseteq A_{n+1} + B_n$. Therefore there exist the cahins $$N = B_0 \subseteq B_1 \subseteq B_2 \subseteq  \dots \quad \text{and} \quad M= A_0 \supseteq A_1 \supseteq A_2 \supseteq \dots$$ such that $A_n \cap B_{n+1} \nsubseteq A_{n+1} + B_n$. The Lemma \ref{l2} shows that $M$ has no Krull dimension and so it has no Noetherian dimension too, a contradiction. Therefore $\ndim\,\frac{M}{N} \leq \alpha$ and this completes the proof.
\end{proof}

It is well-known that every Artinian module $M$ with $J(M)=0$ is also Noetherian. In the next corollary which is now immediate, we generalize this fact to $AB5^*$ modules.

\begin{coro}
    Let $M$ be an $AB5^*$ module with Noetherian dimension. Then $\ndim\,\frac{M}{J(M)} \leq 0$. Moreover, if $J(M) =0$ then $M$ is Noetherian.
\end{coro}

\section{Classification of $\alpha$-short $AB5^*$ modules}

  The concepts of $\alpha$-short and $\alpha$-Krull modules have been studied and investigated respectively in \cite{d-k-sh} and \cite{d-h-sh}. It was proved that if $M$ is an $\alpha$-short module, then $M$ has Noetherian dimension and  either $\ndim\,M = \alpha$ or $\ndim\,M = \alpha + 1$. In particular, a semiprime (non-division) ring $R$ is $\alpha$-short if and only if $\ndim\,R = \alpha$, where $\alpha \geq 0$,  see \cite [Proposition 2.18]{d-k-sh}.
   Here, we should emphasize that  the statement of \cite[Proposition 2.18]{d-k-sh} is not valid for $\alpha = -1$ and it may simply be corrected as the statement of Proposition \ref{p1} in this paper, see also the comment preceding  \cite[Theorem 5.1]{ja-sh}.  However from these, it naturally raised the question of  characterization the $R$-modules $M$ that are $\alpha$-short if and only if $\ndim\,M = \alpha$. We recall that  \cite[Theorem 5.2]{ja-sh} partially has settled this question. It asserts that a semiprime $FQS$-module $M$ is $\alpha$-short if and
   only if $\ndim\,M = \alpha$, where $\alpha \geq 0$.  Here, by an $FQS$-module, we mean a  module $M$ which is finitely generated, quasi-projective and self-generator. Also, a module $M$ is called semiprime if $0$ is a semiprime submodule
   of $M$ in the sense of \cite{ja-sh}. Motivated by what stated and similar to our work in \cite{ja-sh-2}, in the continue of this paper, we will classify the $\alpha$-short modules in the category of $AB5^*$.\\

   We cite the following  definition and results from \cite{d-k-sh}.

\begin{defi}
     An $R$-module $M$ is called to
    be $\alpha$-short if for every submodule $N$ of $M$ either $\ndim\,N \leq \alpha$ or
    $\ndim\, \frac {M}{N} \leq \alpha $ and $\alpha$ is the least ordinal with this property.
\end{defi}

 It is easy to see that $(-1)$-short modules are just simple modules and $0$-short modules are just short modules, in the sense of \cite{b-s}.

 \begin{lem}\label{l3}
    If $M$ is an $\alpha$-short $R$-module, then any submodule and any factor module of $M$ is $\beta$-short for some ordinal $\beta \leq \alpha$.
 \end{lem}

\begin{theo}\label{t3}
    If $M$ is an $\alpha$-short $R$-module, then either $\ndim\,M =\alpha$ or $\ndim\,M =\alpha+1$.
\end{theo}

\begin{theo}\label{t4}
    Let $N$ be a submodule of an $R$-module $M$.
    \begin{enumerate}
        \item  If $N$ is $\alpha$-short and $\ndim\,\frac{M}{N}\leq\alpha$, then $M$ is $\alpha$-short.
        \item If $\frac{M}{N}$ is $\alpha$-short and $\ndim\,N\leq\alpha$, then $M$ is $\alpha$-short.
    \end{enumerate}
\end{theo}

The following result is the corrected statement of \cite [Proposition 2.18]{d-k-sh} that we mentioned before.

\begin{pro}\label{p1}
    A ring $R$ is $\alpha$-short if and only if it is a division ring, where $\alpha = -1$ and when $R$ is a semiprime non-division ring, it is $\alpha$-short if and only if
    $\ndim\, R = \alpha$, where $ \alpha \geq 0$.
\end{pro}

 \begin{proof}
 The proof of the first part follows from the definition and that of the second part is exactly the proof of \cite[Proposition 2.18]{d-k-sh}, without even changing a single word.
\end{proof}

\begin{defi}
    Let $M$ be an $R$-module and $\alpha$ an ordinal number. We set
    $$\mathcal X_\alpha(M) = \{A \subseteq M : \ndim\,\frac{M}{A} \leq \alpha\}, \quad X_\alpha(M) = \bigcap \mathcal X_\alpha(M)$$
    $$\mathcal Y_\alpha(M) = \{B \subseteq M:  \ndim\,B > \alpha\}, \quad Y_\alpha(M) = \bigcap \mathcal Y_\alpha(M).$$
\end{defi}

  Note that in case $\ndim\,M = \alpha$, we have  $X_\alpha(M) = 0$ and $Y_\alpha(M) = M$ and if  $M$ is an $AB5^*$ module, then it follows by Theorem \ref{t2} that $\ndim\,\frac{M}{X_\alpha(M)} \leq \alpha$. The next result is now immediate.

\begin{lem}
    Let $M$ be an $AB5^*$ module.  If $\ndim\,M > \alpha$, then $ Y_\alpha(M) \subseteq X_\alpha(M)$.
\end{lem}

\begin{proof}
It sufficies to show that $\mathcal X_\alpha(M) \subseteq \mathcal Y_\alpha(M)$. To show this, let $A \in \mathcal X_\alpha(M)$. Since $\ndim\,\frac{M}{A} \leq \alpha$ and $\ndim\,M = sup\{\ndim\,A, \ndim\,\frac{M}{A}\} $, we have $\ndim\,A > \alpha$. This shows that $A \in \mathcal Y_\alpha(M)$. So $\mathcal X_\alpha(M) \subseteq \mathcal Y_\alpha(M)$ and we are done.
\end{proof}

In view of the previous lemma and Theorem \ref{t3}, we have the following.

\begin{coro}
    Let $M$ be an $\alpha$-short  $AB5^*$ module. If $ Y_\alpha(M) \nsubseteq X_\alpha(M)$ then $\ndim\,M = \alpha$.
\end{coro}

The following important result is our structure theorem for $\alpha$-short $AB5^*$ modules and plays the major role in achieving what we seek.

\begin{theo}\label{t5}
    Let $M$ be an $AB5^*$ module and $\alpha$ an ordinal number. The following statements are equivalent.

    \begin{enumerate}
        \item  $M$ is an $\alpha$-short module.
        \item  $M$ has a submodule $A$ such that $\ndim\,\frac{M}{A} \leq \alpha$ and $\ndim\,N \leq\alpha$, for any submodule $N\nsupseteq A$ and $\alpha$ is the least ordinal with this property.
    \end{enumerate}
\end{theo}

\begin{proof}
    Let $M$ be $\alpha$-short  and $ A = X_\alpha(M)$. Then
    $\ndim\,\frac{M}{A} \leq \alpha$ by Theorem \ref{t2}. Now let $N$ be a submodule of $M$ such that $A \nsubseteq n$,
    then $\ndim\,\frac{M}{N} \nleq \alpha$ and since $M$ is $\alpha$-short, we have $\ndim\,N \leq \alpha$. In order to show that $\alpha$ is the least ordinal with this property, it sufficies to prove the converse. Conversely, let $M$ has a submodule $A$ with mentioned conditions. For any submodule $N$ of $M$, if $A\subseteq N$, then $\ndim\,\frac{M}{N} \leq  \ndim\,\frac{M}{A} \leq \alpha$ and if $A \nsubseteq N$, then $\ndim\,N \leq \alpha$. Since $\alpha$ is the least ordinal with this property, $M$ is $\alpha$-short.
\end{proof}

\begin{nota}
    We use the following notation in the continue of the paper.
\begin{enumerate}
    \item
    $\mathcal M$ = the set of all modules with Noetherian dimension.
    \item
    $\mathcal M_\alpha $ = the set of all $\alpha$-short modules in  $\mathcal M$.
    \item
    $\mathcal A_\alpha $ = the set of all $M \in \mathcal M_\alpha$ with  $\ndim\,M =\alpha$
    and $\mathcal A = \cup_{\alpha} \mathcal A_\alpha$.
    \item $\mathcal B_\alpha $ = The set of all $M \in \mathcal M_\alpha$ with $\ndim\,M =\alpha + 1$
    and  $\mathcal B = \cup_{\alpha} \mathcal B_\alpha$.
\end{enumerate}
\end{nota}

\begin{rem}
    It is easy to see that
    $\mathcal A_{\alpha} \cap \mathcal
    B_{\alpha}=\emptyset$ and $\mathcal A_\alpha \cup \mathcal
    B_{\alpha} =\mathcal M_{\alpha} $ and so $\{\mathcal A_{\alpha}
    ,\mathcal B_{\alpha} \}$ is a partition for $\mathcal M_{\alpha}$.
    Similarly, $\mathcal A \cap \mathcal B =\emptyset$ and $\mathcal M
    =\mathcal A \cup \mathcal B$ so $\{\mathcal A,\mathcal B\}$ is a
    partition for $\mathcal M$.
\end{rem}

\begin{theo}\label{t6}
    Let $M \in \mathcal {B_\alpha}$ and $N$ be a submodule of $M$.  Then either $N \in \mathcal{B_\alpha}$ or $\frac{M}{N} \in \mathcal{B_\alpha}$.
\end{theo}

\begin{proof}
    By the above notation, $M$ is $\alpha$-short and $\ndim\,M = \alpha
    +1$. We have the following cases.
    \begin{enumerate}
        \item If $\ndim\,N > \ndim\,\frac{M}{N}$, then $\ndim\,N = \ndim\,M = \alpha + 1$. Also $N$ is $\beta$-short for some $\beta \leq \alpha$. If $\beta < \alpha$,  by Theorem \ref{t3}, we have $\ndim\,N \leq \beta + 1 \leq \alpha$, a contradiction. Consequently,  $N$ is $\alpha$-short with $\ndim\,N = \alpha + 1$. Therefore $N \in \mathcal{B}_\alpha$.
        \item  If $\ndim\,\frac{M}{N} > \ndim\,N$, then $\ndim\,\frac{M}{N} = \ndim\,M = \alpha + 1$. Also $\frac{M}{N}$ is $\beta$-short for some $\beta \leq \alpha$. If $\beta < \alpha$, by Theorem \ref{t3}, we have $\kdim\,\frac{M}{N} \leq \beta + 1 \leq \alpha$, a contradiction. Consequently,  $\frac{M}{N}$ is $\alpha$-short with $\ndim\,\frac{M}{N} = \alpha + 1$. Therefore $\frac{M}{N} \in \mathcal{B}_\alpha$. At last,
        \item If $\ndim\,N = \ndim\,\frac{M}{N}$, then the same argument shows that both $N$ and $\frac{M}{N}$ belong to $\mathcal{B}_\alpha$. Hence, we are done.
    \end{enumerate}
\end{proof}

In view of the proof of Theorem \ref{t6} the next result is now immediate.

\begin{coro}\label{c1}
    Let $M\in \mathcal{B_\alpha}$ and $N$ be a submodule of $M$.  Then
    \begin{enumerate}
        \item  If $\ndim\,N > \ndim\,\frac{M}{N}$, then $N \in \mathcal{B_\alpha}$.
        \item  If $\ndim\,N < \ndim\,\frac{M}{N}$, then $\frac{M}{N}\in\mathcal{B_\alpha}$.
        \item  If $\ndim\,N = \ndim\,\frac{M}{N}$, then $N , \frac{M}{N} \in\mathcal{B_ \alpha}$.
    \end{enumerate}
\end{coro}

\begin{coro}\label{c2}
    Let $N$ be a submodule of $M$.
    If $N , \frac{M}{N} \in \mathcal{A}$, then $M \in \mathcal{A}$.
\end{coro}

\begin{coro}\label{c3}
    If $M_1, M_2 ,. . . , M_n \in \mathcal{A}$, then $M_1\oplus M_2\oplus ...\oplus M_n \in \mathcal{A}$.
\end{coro}

\begin{rem}
    The converse of Corollary \ref{c3} is not true in general. For example, if $M = M_1 \oplus M_2$ where
    $M_1$ and $M_2$ are simple modules, then $M_1 , M_2 \in \mathcal{B}_{-1}$.  But $M \in \mathcal{M}$ and it is also a (non-simple) semisimple module. This shows that $ M\in\mathcal{A}$. Note that any semisimple module $M \in \mathcal M$ is both a short and Noetherian module.
\end{rem}

\begin{coro}\label{c4}
    If $M\in\mathcal{B_\alpha}$,  then $\frac{M}{E} \in\mathcal{B_\alpha}$ for some  essential submodule $E$ of $M$.
\end{coro}

\begin{proof}
    Since $ M \in \mathcal{B}_\alpha $,  $\ndim\,M = \alpha + 1 > 0 $, we conclude that $M$ is not semisimple (note, a semisimple modules with Noetherian dimension is both Artinian and Noetherian). Hence, $M$ has proper essential submodules. We have
    $\ndim\,M = sup\{\ndim\,\frac{M}{E} : E \quad \text{is an essential submodule of} \quad M\}$,
    so $\ndim\,\frac{M}{E}=\alpha+1$  for some essential submodule $E$ of $M$. By Theorem \ref{t4}, $\frac{M}{E}\in\mathcal{B_\alpha}$
\end{proof}

 Recall that by  an $\alpha$-atomic module we mean a non-zero module $M$ such that $\ndim\,M = \alpha$ and $\ndim\,N < \alpha$ for every proper submodule $N$ of $M$.

\begin{theo}\label{t7}
    Let $M$ be an $\alpha$-short $AB5^*$ module. The following statements are equivalent.
    \begin{enumerate}
        \item $M\in\mathcal{B_\alpha}$.
        \item  $X_\alpha(M)$ is $(\alpha+1)$-atomic.
    \end{enumerate}
\end{theo}

\begin{proof}
    $(1) \rightarrow (2)$ Since  $M\in\mathcal{B_\alpha}$, we have $\ndim\,M=\alpha+1$. By Theorem \ref{t2},  $\ndim\,\frac{M}{X_\alpha(M)} \leq \alpha$.
    This implies that $X_\alpha(M)\ne 0$ and $\ndim\,X_\alpha(M)=\alpha+1$. Let $B\subsetneq X_\alpha(M)$.  By Theorem \ref{t5},
    $\ndim\,B \leq \alpha$ and so $X_\alpha(M)$ is $(\alpha+1)$-atomic.\\
    $(2) \rightarrow (1)$ Let $X_\alpha(M)$ be $(\alpha+1)$-atomic, then $X_\alpha(M)$ is clearly $\alpha$-short and by Theorem \ref{t5},
    $\ndim\,\frac{M}{X_\alpha(M)} \leq \alpha$. Hence $M$ is $\alpha$-short, by Theorem \ref{t4}.
\end{proof}

The next theorem determines $AB5^*$ modules, belong to $\mathcal{B}$.

\begin{theo}\label{t8}
    Let $M$ be an  $AB5^*$ module. The following statements are equivalent.
    \begin{enumerate}
        \item  $M \in \mathcal{B}$
        \item  $\ndim\, M$ is not a limit ordinal and $M$ has an atomic submodule $A$ with
        $ \ndim\,\frac{M}{A} < \ndim\,A $.
    \end{enumerate}
\end{theo}

\begin{proof}
    $(1)\rightarrow (2)$ If $M \in \mathcal{B}$, then there exists an ordinal  $\alpha$ such that $M \in\mathcal{B_\alpha}$,
    so $\ndim\,M =\alpha+1$ is not a limit ordinal. Also, $A=X_\alpha(M)$ is an $(\alpha+1)$-atomic submodule of $M$ by Theorem \ref{t7}. Moreover, $\ndim\,\frac{M}{A} \leq\alpha<\alpha+1=\ndim\,A$, by Theorem \ref{t5}.
    $(2)\rightarrow(1)$ Let $\ndim\,M = \alpha +1$ and $A$ be an atomic submodule of $M$
    such that $\ndim\,\frac{M}{A}<\ndim\,A$,  then  $ \ndim\,A = \alpha + 1$. But $A$ is $\alpha$-short and so is $M$, by theorem \ref{t4}. Therefore $M \in \mathcal{B}$.
\end{proof}

The next important result is now immediate.

\begin{theo}
    Let $M$ be an  $AB5^*$ module. The following statements are equivalent.
    \begin{enumerate}
        \item $M \in \mathcal{A}$.
        \item $\ndim\,M$ is a limit ordinal or $\ndim\,A \leq \ndim\,\frac{M}{A}$,
        for any atomic submodule $A$ of $M$.
    \end{enumerate}
\end{theo}



\end{document}